\documentclass[12pt]{article}

\usepackage{hyperref}

\usepackage{amsfonts, amsmath, amssymb, amsthm}
\usepackage{a4}
\usepackage{graphicx}

\DeclareMathOperator{\dethad}{dh}

\newtheorem{proposition}{Proposition}

\theoremstyle{remark}

\title{Quadratic heptagon cohomology}
\author{Igor G. Korepanov}

\date{October 2021}

\begin{document}

\sloppy

\maketitle

\begin{flushright}{\it To the memory of Aristophanes Dimakis}\end{flushright}

\medskip

\begin{abstract}
A cohomology theory is proposed for the recently discovered heptagon relation---an algebraic imitation of a 5-dimen\-sional Pachner move 4--3. In particular, `quadratic cohomology' is introduced, and it is shown that it is quite nontrivial, and even more so if compare heptagon with either its higher analogues, such as enneagon or hendecagon, or its lower analogue, pentagon. Explicit expressions for the nontrivial quadratic heptagon cocycles are found in dimensions 4 and~5.
\end{abstract}

\section{Introduction}\label{s:i}

\subsection{Heptagon and cohomology}\label{ss:hc}

This paper is the continuation of paper~\cite{hepta_1} where the first nontrivial heptagon relation was found. We recall here that by heptagon relation, we mean an algebraic imitation of a Pachner move 4--3 in a triangulation of a five-dimen\-sional piecewise linear (PL) manifold, see~\cite{Pachner,Lickorish} for Pachner moves and~\cite{hepta_1,DM-H,DK,cubic, nonconstant} for heptagon and other polygon relations. Discovering heptagon relations may be viewed as a preparational algebraic work for constructing invariants of these manifolds and related topological field theories.

It is known, however, that often more refined invariants can be obtained if we use not only a relation like heptagon but also its \emph{cohomology}. Enough to mention, as an example, invariants of knots and knotted surfaces coming from just \emph{quandles} and from \emph{quandle cohomology}, reviewed in~\cite{CKS}.

It can be expected that, for five-dimen\-sional PL manifolds, \emph{heptagon 5-cocycles} will play the principal role. As explained in~\cite{hepta_1}, we \emph{color} the 4-faces of a triangulation with elements of a field~$F$, and declare some of the colorings for (the faces of) each 5-simplex \emph{permitted}. Roughly speaking, heptagon itself is expected to give, as a manifold~$M$ invariant, the dimension of the vector space of its triangulation's permitted colorings, corrected by some simple multiplier. Different values~$c[M]$ of a heptagon 5-cocycle~$c$ on~$M$ can add more information to that invariant.

Field~$F$ must be big enough to allow for some sets of its elements to be in a `general position', see, for instance, an assumption made below right after formula~\eqref{dijk}. When speaking of quadratic cohomology, we also assume that its characteristic is not~2. Moreover, for some of our calculations, $F$ must be of characteristic~0, which is explicitly indicated in the relevant places in the text.

The heptagon relation and its cohomology studied here are \emph{nonconstant}. This means that the specific objects belonging to a simplex (such as e.g.\ the matrix determining its permitted colorings or the coefficients in the expression for a cocycle) depend on parameters that are different for different simplices. The difference between `constant' and `nonconstant' cohomology theories can be seen on the \emph{hexagon} example by comparing papers \cite{cubic} and~\cite{nonconstant}.

\subsection{The main result}\label{ss:mr}

The main result of this paper is the existence of a non-trivial 5-cocycle for the heptagon relation introduced in~\cite{hepta_1}, together with explicit algebraic formulas for it. This looks actually quite striking if compared with the apparent non-existence of a similar cocycle for other odd polygon relations of the same kind. Namely, we analyze here also pentagon, enneagon (9-gon) and hendecagon (11-gon).

Recall that the existence of higher (than pentagon and heptagon) odd polygon relations was announced in~\cite[Subsection~6.3]{hepta_1}, together with a brief but explicit explanation of the key step for their construction. We still leave their general analysis for future works, but as we want to compare heptagon with pentagon, enneagon and hendecagon, we simply write out  necessary formulas for these, and content ourself with the fact that, in these cases, the formulas (including polygon relations as such) can be checked directly using computer algebra.

Anyhow, we could not do without computer algebra altogether, even for the heptagon, both when searching for our 5-cocycle and when checking its validity. This probably means that there are hidden algebraic structures, yet to be discovered.

\subsection{Notation changes with respect to paper~\cite{hepta_1}}\label{ss:ch}

Aristophanes Dimakis taught me a more elegant, than in~\cite{hepta_1}, way of vertex and simplex numbering in polygon relations, used in~\cite{DM-H,DK}. Accordingly, our Pachner move now replaces the cluster of 5-simplices with numbers 1, 3, 5 and 7 with the cluster of 5-simplices 2, 4 and~6. The heptagon looks now as in Figure~\ref{fig:hepta-new}
\begin{figure}
 \begin{center}
  \includegraphics[scale=1.25]{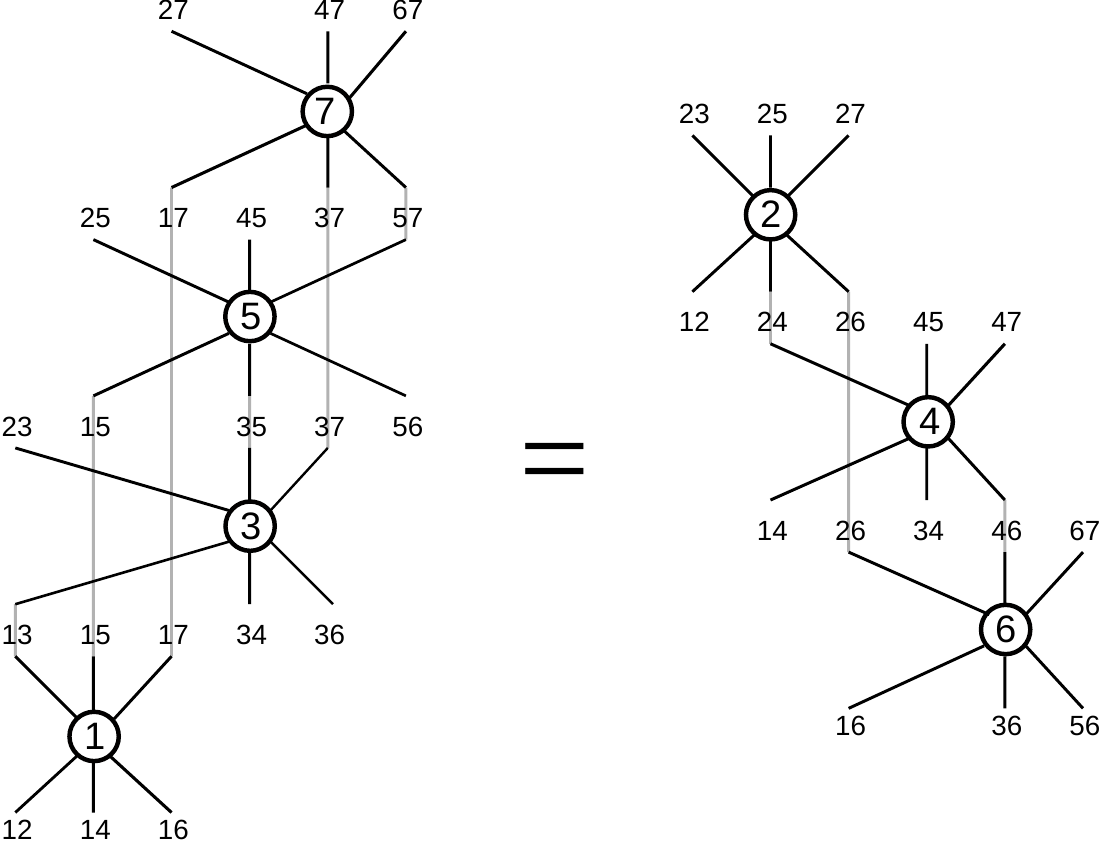}
 \end{center}
 \caption{Heptagon relation}
 \label{fig:hepta-new}
\end{figure}
rather than~\cite[Figure~1]{hepta_1}, and reads
\begin{equation}\label{hepta}
A_{123}^{(1)}A_{145}^{(3)}A_{246}^{(5)}A_{356}^{(7)}=A_{356}^{(6)}A_{245}^{(4)}A_{123}^{(2)}\,,
\end{equation}
instead of~\cite[Eq.~(1)]{hepta_1}. Note that the \emph{subscripts} in~\eqref{hepta} did \emph{not} change.

Also, we change the meaning of letter~$n$, to bring it in conformity with~\cite{DK}: our odd polygons are now $(2n+1)$-gons, so $n=2$ for pentagon, $n=3$ for heptagon, and so on.

\subsection{Contents of the rest of the paper}\label{ss:cr}

Below,
\begin{itemize} \itemsep 0pt
 \item in Section~\ref{s:h}, we recall our heptagon relation introduced in~\cite{hepta_1},
 \item in Section~\ref{s:coho}, we introduce both a general cohomology theory for heptagon and its specific `quadratic' version,
 \item in Section~\ref{s:4c}, we present a quadratic cocycle in dimension~4, and show its uniqueness,
 \item in Section~\ref{s:5c}, we present a quadratic cocycle in dimension~5 together with some related reasonings and observations,
 \item in Section~\ref{s:o}, we find out that, while analogues of the heptagon 4-cocycle exist for pentagon, enneagon and hendecagon, such analogues of the heptagon 5-cocycle do not exist,
 \item finally, in Section~\ref{s:d}, we discuss possible directions of further research.
\end{itemize}

\section{Heptagon relation}\label{s:h}

Here we recall some conventions and facts from~\cite{hepta_1}.

\subsection{Notations}\label{ss:nota}

Recall that~\eqref{hepta} and Figure~\ref{fig:hepta-new} mean the equalness of two products of matrices acting in the direct sum of six copies of field~$F$. We represent this direct sum as the 6-\emph{row} space, so our matrices act on \emph{rows} from the~\emph{right}. Each matrix~$A_{abc}^{(p)}$ acts nontrivially only in the copies (or: on the row elements) number $a$, $b$ and~$c$. Circles with numbers in Figure~\ref{fig:hepta-new} depict the seven 5-simplices taking part in move 4--3, while edges correspond to \emph{4-faces}: edge~$ij$ depicts the 4-face common for simplices $i$ and~$j$, including the case where these belong to the different sides of Figure~\ref{fig:hepta-new} (recall that the lhs and rhs of a Pachner move have the \emph{same} boundary). The order of $i$ and~$j$ is here, by definition, irrelevant: edge/face~$ij$ is the same as~$ji$.

It proved also convenient for us to denote the \emph{vertices} of simplices taking part in the Pachner move by the same letters (or numbers) as 5-simplices, using the following principle: 5-simplex~$i$ has all vertices $1,\ldots,7$ \emph{except}~$i$. That is, 5-simplex~$1$ can also be denoted as~$234567$; accordingly, 4-face~$12$ is the same as~$34567$.

Note that below, in Subsection~\ref{ss:ev}, we speak of \emph{edges of these simplices}, not to be confused with edges of Figure~\ref{fig:hepta-new}.

\subsection{Explicit expression for matrix entries}\label{ss:em}

Explicitly, entries of matrices~$A_{abc}^{(p)}$ are described as follows. First, we introduce a matrix
\begin{equation}\label{Gelm}
\mathcal M = 
\begin{pmatrix} \alpha_1 & \alpha_2 & \alpha_3 & \alpha_4 & \alpha_5 & \alpha_6 & \alpha_7 \\
                \beta_1 & \beta_2 & \beta_3 & \beta_4 & \beta_5 & \beta_6 & \beta_7 \\ 
                \gamma_1 & \gamma_2 & \gamma_3 & \gamma_4 & \gamma_5 & \gamma_6 & \gamma_7 
\end{pmatrix}
\end{equation}
with entries in~$F$, and determinants
\begin{equation}\label{dijk}
d_{ijk} = \left| \begin{matrix} \alpha_i & \alpha_j & \alpha_k \\ \beta_i & \beta_j & \beta_k \\ \gamma_i & \gamma_j & \gamma_k \end{matrix} \right|
\end{equation}
made of triples of its columns. We assume that the entries of~$\mathcal M$ are \emph{generic} enough in the exact sense that $d_{ijk}$ does not vanish for any pairwise different $i,j,k$.

Each separate matrix~$A^{(p)}$ looks as in Figure~\ref{fig:Ahepta}.
\begin{figure}
 \begin{center}
  \includegraphics[scale=1.25]{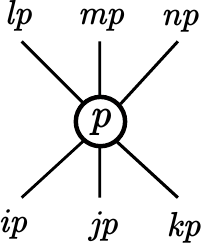}
 \end{center}
 \caption{Matrix~$A^{(p)}$ for heptagon}
 \label{fig:Ahepta}
\end{figure}
By definition, the entry of~$A^{(p)}$ corresponding to the input (lower) leg $ip=pi$ and output (upper) leg $lp=pl$ is
\begin{equation}\label{ansatz-h}
\left( A^{(p)} \right)_{ip}^{lp} = \frac{d_{jlp}d_{klp}}{d_{ijp}d_{ikp}},
\end{equation}
where $jp$ and~$kp$ are the other input legs of~$A^{(p)}$.

It was proved in~\cite{hepta_1} that matrices~\eqref{ansatz-h} satisfy~\eqref{hepta}.

\subsection{Edge vectors}\label{ss:ev}

The main technical tool used in~\cite{hepta_1} to construct our heptagon relation was \emph{edge vectors}---permitted colorings of the Pachner move such that nonzero colors are ascribed only to 4-faces containing a given edge $b=ij$. Recall that the explicit expression for the component~$e_b|_u$ of edge vector~$e_b$ corresponding to 4-face~$u$ is
\begin{equation}\label{dd}
e_{ij}|_u = d_{ilm}d_{jlm},
\end{equation}
where $l$ and~$m$ are the two vertices \emph{not} belonging to~$u$ (while $i$ and~$j$, of course, do belong).

Any four vectors $e_{ij}, e_{ik}, e_{il}, e_{im}$ corresponding to edges with a common vertex~$i$ are linearly dependent:
\begin{equation}\label{lmnx}
\lambda_{i,ij}^{(ijklm)} e_{ij}+\lambda_{i,ik}^{(ijklm)} e_{ik}+\lambda_{i,il}^{(ijklm)} e_{il}+\lambda_{i,im}^{(ijklm)} e_{im}=0 ,
\end{equation}
where~$\lambda_{i,ij}^{(ijklm)}$ are some numbers whose explicit expression is given in~\cite[Eq.~(13)]{hepta_1}.

\section{Cohomology theory}\label{s:coho}

\subsection{Nonconstant heptagon cohomology: generalities}\label{ss:h}

In this subsection, we define ``nonconstant heptagon cohomology'' in a general context.

Our definition will depend on a chosen simplicial complex~$K$. In principle, $K$ can be of any dimension, although the main work in this paper will take place in the standard 6-simplex~$K=\Delta^6$, whose boundary is the union of the lhs and rhs of any five-dimen\-sional Pachner move.

Suppose that every 4-simplex~$u\subset K$ is colored by some element $\mathsf x_u\in X$ of a set~$X$ of colors (for instance, our field $X=F$), and that a subset~$R_v$ of permitted colorings is defined in the set of all colorings of every 5-simplex~$v$ (for instance, by declaring that the ``output'' colors must be determined by the ``input'' ones, using matrix~\eqref{ansatz-h}).

We also define the set of permitted colorings for any simplex of~$K$ of dimension~$>5$: the coloring is permitted provided its restrictions on all 5-faces of that simplex are permitted. As for an individual 4-simplex, \emph{all} its colorings $\mathsf x\in X$ are permitted by definition.

The set of all permitted colorings of an $m$-simplex $i_0\dots i_m$ will be denoted~$\mathfrak C_{i_0\dots i_m}$. We assume here that the vertices of any simplex are ordered: $i_0<\ldots <i_m$.

Let an abelian group~$G$ be given. By definition, an \emph{$m$-cochain}~$\mathfrak c$ taking values in~$G$, for $m\ge 3$, consists of arbitrary mappings
\begin{equation}\label{nc}
\mathfrak c_{i_0\dots i_m}\colon\;\,\mathfrak C_{i_0\dots i_m} \to G
\end{equation}
for \emph{all} $m$-simplices $\Delta^m=i_0\dots i_m\subset K$.

The \emph{coboundary}~$\delta \mathfrak c$ of~$\mathfrak c$ consists then of mappings $(\delta \mathfrak c)_{i_0\dots i_{m+1}}$ acting on a permitted coloring~$r\in \mathfrak C_{i_0\dots i_{m+1}}$ of $(m+1)$-simplex $i_0\dots i_{m+1}$ according to the following formula:
\begin{equation}\label{cb}
(\delta \mathfrak c)_{i_0\dots i_{m+1}} (r) = \sum_{k=0}^{m+1} (-1)^k\, \mathfrak c_{i_0\dots \widehat{i_k} \dots i_{m+1}} (r|_{i_0\dots \widehat{i_k} \dots i_{m+1}}),
\end{equation}
where each $r|_{i_0\dots \widehat{i_k} \dots i_{m+1}}$---the restriction of~$r$ onto the $m$-simplex $i_0\dots \widehat{i_k} \dots i_{m+1}$---is of course a permitted coloring of this latter simplex.

Our `nonconstant heptagon cohomology' is the cohomology of the following \emph{heptagon cochain complex}:
\begin{equation}\label{hcc}
0 \to C^4 \stackrel{\delta}{\to} C^5 \stackrel{\delta}{\to} C^6 \stackrel{\delta}{\to} \dots\, ,
\end{equation}
where $C^m$ means the group of all $m$-cochains.

\subsection{Quadratic cohomology}\label{ss:q}

It turns out that there are some interesting variations of the cochain definition~\eqref{nc}. For instance, preprint~\cite{cubic} (although devoted to \emph{constant hexagon} cohomology) suggests that \emph{homogeneous polynomials} of a given degree may be used instead of general functions~\eqref{nc}---of course, in a situation where the notion of polynomial in the variables determining a permitted coloring makes sense.

Of special interest for us in this paper will be \emph{quadratic} cochains. That is, let $G=F$ be the same field~$F$ that we are using for the set of colors, and mappings~$\mathfrak c_{i_0\dots i_m}$ in~\eqref{nc} be quadratic forms on the linear spaces of permitted colorings of corresponding $m$-simplices~$i_0\ldots i_m$.

Remember (Subsection~\ref{ss:hc}) that our field~$F$ is of characteristic~$\ne 2$. Hence, there is a \emph{polarization} for any quadratic cochain. By definition, this is the \emph{symmetric bilinear cochain} depending on \emph{two} (independent from each other) permitted colorings:
\begin{equation}\label{bc}
\mathfrak c_{i_0\dots i_m}^{\mathrm{pol}}\colon\quad \mathfrak C_{i_0\dots i_m} \times \mathfrak C_{i_0\dots i_m} \to F,
\end{equation}
where each $\mathfrak c_{i_0\dots i_m}^{\mathrm{pol}}$ is the polarization of the corresponding~$\mathfrak c_{i_0\dots i_m}$ in~\eqref{nc}.

Mapping~\eqref{bc} can be treated as a \emph{scalar product} of two permitted colorings of simplex~$i_0\dots i_m$. So, to define a quadratic cochain is essentially the same as to define a scalar product for all corresponding simplices. In what follows, we will make extensive use of this fact.

\section{Quadratic 4-cocycle}\label{s:4c}

Here and in the next Section~\ref{s:5c}, we are working within the simplicial complex $K=\Delta^6=1234567$, and even within its boundary~$\delta\Delta^6$ which consists of seven 5-simplices and is the union of the two parts of a five-dimen\-sional Pachner move.

Let $x_{ip}$ denote the color of face~$ip$. Recall that $ip$ means, in this context, the 4-face containing all vertices $1,\ldots,7$ \emph{except} $i$ and~$p$.

\begin{proposition}\label{p:4c}
The cochain consisting of mappings
\begin{equation}\label{4c}
x_{ip} \mapsto c_{ip\,} x_{ip}^2,
\end{equation}
where
\begin{equation}\label{cip}
 c_{ip}=\prod_{\substack{\iota_1 \ne i,p\\ \iota_2 \ne i,p\\ \iota_1 < \iota_2}} d_{\iota_1 \iota_2 p} ,
\end{equation}
is a quadratic 4-cocycle.
\end{proposition}

Geometrically, the product in~\eqref{cip} goes over the ten \emph{edges}~$\iota_1 \iota_2$ of 4-face~$ip$.

\medskip

We introduce the following scalar product of two arbitrary colorings $x$ and~$y$ of a 5-simplex~$p$:
\begin{equation}\label{xy4}
\sum_{i\ne p} \epsilon_i^{(p)} c_{ip}\,x_{ip}y_{ip} = \langle x, y \rangle_4^{(p)},
\end{equation}
where $\epsilon_i^{(p)}=\pm 1$ is the alternating sign on the set $\{1,\ldots,7\}\setminus \{p\}$ taken in the increasing order: the first element has sign~$+$, the second---sign~$-$, \ldots, the $i$th---sign~$\epsilon_i^{(p)}$. For \emph{permitted} $x$ and~$y$ this is, of course, just the coboundary of~\eqref{4c} for 5-simplex~$p$, hence, what we must prove is that \eqref{xy4} vanishes for permitted $x$ and~$y$.

Pay attention to the subscript~$4$ in the rhs of~\eqref{xy4}: it serves to distinguish this product from another one introduced below in~\eqref{q5}.

\begin{proof}[Proof of Proposition~\ref{p:4c}]
Take, first, edge vectors for two non-intersecting edges as $x$ and~$y$. For instance, consider 5-simplex $123456$ (that is, $p=7$), and let $x=e_{12}$ and $y=e_{34}$. Then, there are just two nonvanishing summands ($i=5$ and~$6$) in the lhs of~\eqref{xy4}, and they are easily seen to cancel each other.

Then we note that only \emph{restrictions} of colorings onto 5-simplex~$p$ take part in~\eqref{xy4}, and if we put $m=p$ in~\eqref{lmnx}, and restrict the lhs of~\eqref{lmnx} onto~$p$, then there remain only \emph{three} terms in the obtained linear dependence---because vertex~$p$ and hence edge~$ip$ do \emph{not} belong to the 5-simplex~$p$, see the second paragraph of Subsection~\ref{ss:nota}.

Now it is not hard to see that we can use such three-term linear dependences to show that \eqref{xy4} vanishes for \emph{any} edge vectors $x$ and~$y$.
\end{proof}

\begin{proposition}\label{p:4cnm}
There are no 4-cocycles linearly independent from cocycle given by \eqref{4c} and~\eqref{cip}.
\end{proposition}

\begin{proof}
We note first what follows, for instance, from the fact that $\langle e_{12}, e_{34}\rangle_4^{(7)}$ must vanish. Recall that `$7$' means here the 5-simplex~$123456$, and it has only two 4-faces containing edges $12$ and~$34$ at once---$12345$ and~$12346$, or, in other notations (see Subsection~\ref{ss:nota}), $67$ and~$57$. Hence, the vanishing of $\langle e_{12}, e_{34}\rangle_4^{(7)}$ determines $c_{57}/c_{67}$ uniquely. Similarly, all ratios between coefficients~$c_{ip}$ are also determined uniquely.
\end{proof}

\section{Quadratic 5-cocycle}\label{s:5c}

\subsection{Why a nontrivial quadratic 5-cocycle must exist}\label{ss:5c-why}

We first calculate the numbers of linearly independent quadratic 4-, 5- and 6-cochains in our simplicial complex $K=\Delta^6=1234567$.

\paragraph{4-cochains}
There are 21 linearly independent 4-cochains, one for each 4-dimen\-sional face~$ip$, namely cochains~$x_{ip}^2$.

\paragraph{5-cochains}
For each 5-simplex, the linear space of permitted colorings is 3-dimen\-sional: three arbitrary ``input'' colors determine three ``output'' ones. The space of quadratic forms of 3 variables is 6-dimen\-sional. And there are seven 5-simplices (that is, simplices with \emph{six} vertices!) in the heptagon relation.

Hence, there is the $7\times 6=42$-dimen\-sional linear space consisting of 7-tuples of quadratic forms, one for each 5-simplex.

\paragraph{6-cochains}
There are six independent ``input'' colors for the whole heptagon (see again Figure~\ref{fig:hepta-new}), so there are $\frac{6\times 7}{2}=21$ linearly independent quadratic forms.

\medskip

We now write out a fragment of sequence~\eqref{hcc} for quadratic cochains, \emph{assuming that the characteristic of our field~$F$ is zero}:
\begin{equation}\label{456}
\begin{pmatrix} 21\\ \text{4-cochains} \end{pmatrix} \xrightarrow[\mathrm{rank} = 20]{\textstyle\delta} \begin{pmatrix} 42\\ \text{5-cochains} \end{pmatrix} \xrightarrow[\mathrm{rank} = 21]{\textstyle\delta} \begin{pmatrix} 21\\ \text{6-cochains} \end{pmatrix} .
\end{equation}
Here follows the explanations.

First, ``21 cochains'' stays in~\eqref{456} for ``21-dimen\-sional space of cochains'', and so on.

Second, the rank of the left coboundary operator~$\delta$ in~\eqref{456} is $20$ and not~$21$ because there exists exactly one nonzero 21-tuple that gives exactly zero on each 5-simplex, according to Propositions \ref{p:4c} and~\ref{p:4cnm}.

Third, the rank of the \emph{right} operator~$\delta$ is surely $\le 21$, and this is already enough to conclude that the cohomology space dimension in the middle term is $\ge 42-20-21=1$. Actually, a direct calculation, made in characteristic~$0$, shows that the rank of the right~$\delta$ is exactly~$21$ for generic matrices~$\mathcal M$~\eqref{Gelm}, but is looks, at this moment, more difficult to understand why it is so. It is also unknown whether the assumption made right after formula~\eqref{dijk} is enough to guarantee that~$\mathcal M$ is generic in this sense.

\subsection{Explicit form of 5-cocycle}\label{ss:xpl5}

We define now one more scalar product, this time between two \emph{permitted} colorings of a 5-simplex~$p$. For the case where these colorings are the restrictions of edge vectors $e_{ij}$ and $e_{kl}$, respectively, on~$p$, we set
\begin{equation}\label{q5}
\langle e_{ij}, e_{kl} \rangle_5^{(p)} \, \stackrel{\mathrm {def}}{=} \, \det\eta_p \cdot (d_{ikp}d_{jlp}+d_{ilp}d_{jkp}),
\end{equation}
where
\begin{equation}\label{eta}
\eta_p = 
\begin{pmatrix} \alpha_i^2 & \alpha_j^2 & \alpha_k^2 & \alpha_l^2 & \alpha_m^2 & \alpha_n^2 \\[.5ex]
                \beta_i^2 & \beta_j^2 & \beta_k^2 & \beta_l^2 & \beta_m^2 & \beta_n^2 \\[.5ex]
                \gamma_i^2 & \gamma_j^2 & \gamma_k^2 & \gamma_l^2 & \gamma_m^2 & \gamma_n^2 \\[.5ex]
                \alpha_i \beta_i & \alpha_j \beta_j & \alpha_k \beta_k & \alpha_l \beta_l & \alpha_m \beta_m & \alpha_n \beta_n \\[.5ex]
                 \alpha_i \gamma_i & \alpha_j \gamma_j & \alpha_k \gamma_k & \alpha_l \gamma_l & \alpha_m \gamma_m & \alpha_n \gamma_n \\[.5ex]
                 \beta_i \gamma_i & \beta_j \gamma_j & \beta_k \gamma_k & \beta_l \gamma_l & \beta_m \gamma_m & \beta_n \gamma_n
\end{pmatrix} ,
\end{equation}
and $i,\ldots,n$ are the numbers from 1 through~7 \emph{except}~$p$, going in the increasing order.

As edge vectors make not a basis but an \emph{overfull system} of vectors in the linear space of all permitted colorings, the following proposition is necessary to justify this definition.

\begin{proposition}\label{p:corr5}
Formula~\eqref{q5} defines a scalar product in the linear space of permitted colorings of 5-simplex~$p$ correctly.
\end{proposition}

\begin{proof}
It must be checked that the definition~\eqref{q5} agrees with the three-term linear dependences mentioned in the proof of Proposition~\ref{p:4c}. Consider such a linear dependence
\begin{equation}\label{cor1}
\lambda_1 e_{ij_1}|_p + \lambda_2 e_{ij_2}|_p + \lambda_3 e_{ij_3}|_p = 0,
\end{equation}
and, on the other hand, the determinants $d_{jlp}$ and~$d_{jkp}$ entering in the rhs of~\eqref{q5}. It is a simple consequence from the explicit form~\cite[Eq.~(13)]{hepta_1} of lambdas and a Pl\"ucker bilinear relation for determinants that a similar to~\eqref{cor1} relation holds for any of these determinants, for instance,
\begin{equation}\label{cor2}
\lambda_1 d_{j_1 lp} + \lambda_2 d_{j_2 lp} + \lambda_3 d_{j_3 lp} = 0.
\end{equation}
Hence, for the scalar product defined according to~\eqref{q5} we also have the desirable equality showing that definition~\eqref{q5} is self-consistent:
\begin{equation}\label{corr}
\lambda_1 \langle e_{ij_1}, e_{kl} \rangle_5^{(p)} + \lambda_2 \langle e_{ij_2}, e_{kl} \rangle_5^{(p)} + \lambda_3 \langle e_{ij_3}, e_{kl} \rangle_5^{(p)} = 0. 
\end{equation}
\end{proof}

\begin{proposition}\label{p:cr}
Formulas \eqref{q5} and~\eqref{eta} define, indeed, a cocycle:
\begin{equation}\label{5c-usl}
\sum_{p=1}^7 (-1)^p \langle e_{ij}, e_{kl} \rangle_5^{(p)} = 0
\end{equation}
for any two edges $ij$ and~$kl$.
\end{proposition}

\begin{proof}
Direct calculation.
\end{proof}

\subsection{Nontriviality}\label{ss:ntr}

\begin{proposition}\label{p:nt}
Cocycle defined according to \eqref{q5} and~\eqref{eta} is nontrivial---not a coboundary.
\end{proposition}

\begin{proof}
Coboundary is, in this situation, a linear combination of~$x_{ip}^2$ taken over 4-faces. The scalar product $\langle x, y \rangle_5^{(p)}$ corresponding to such a cocycle would then be a linear combination of products~$x_{ip}y_{ip}$. This would imply, taking~\eqref{dd} into account, that $\langle e_{12}, e_{34} \rangle_5^{(7)}$, $\langle e_{13}, e_{24} \rangle_5^{(7)}$ and $\langle e_{14}, e_{23} \rangle_5^{(7)}$ would all three coincide---but they are actually all different.
\end{proof}

\subsection{One more observation}\label{ss:ob}

Matrix~$\mathcal M$~\eqref{Gelm} can be reduced, by a linear transformation of its rows, to the form where its three first columns form an identity matrix. Suppose this has been already done, that is,
\begin{equation}\label{id3}
\begin{pmatrix} \alpha_1 & \alpha_2 & \alpha_3 \\
                \beta_1 & \beta_2 & \beta_3 & \\ 
                \gamma_1 & \gamma_2 & \gamma_3 \end{pmatrix} =
\begin{pmatrix} 1 & 0 & 0 \\ 0 & 1 & 0 \\ 0 & 0 & 1 \end{pmatrix},
\end{equation}
and consider the determinant of matrix~$\eta_7$~\eqref{eta} for such~$\mathcal M$. A calculation shows that
\begin{align}\label{dh}
\det \eta_7 = 
 \alpha_4  \beta_4  \alpha_5  \gamma_5  \beta_6  \gamma_6 -  \alpha_4  \gamma_4  \alpha_5  \beta_5  \beta_6  \gamma_6 -  \alpha_4  \beta_4  \beta_5  \gamma_5  \alpha_6  \gamma_6 \\
+  \beta_4  \gamma_4  \alpha_5  \beta_5  \alpha_6  \gamma_6 +  \alpha_4  \gamma_4  \beta_5  \gamma_5  \alpha_6  \beta_6 -  \beta_4  \gamma_4  \alpha_5  \gamma_5  \alpha_6  \beta_6 \\
\stackrel{\mathrm {def}}{=} -\dethad
\begin{pmatrix} \alpha_4 & \alpha_5 & \alpha_6 \\
                \beta_4 & \beta_5 & \beta_6 & \\ 
                \gamma_4 & \gamma_5 & \gamma_6 \end{pmatrix} ,
\end{align}
where function `$\dethad$' on $3\times 3$ matrices was introduced in~\cite[Eq.~(7)]{igrushka} in connection with what seemed a completely different problem---evolution of a discrete-time dynamical system, where one step of evolution consisted in taking, first, the usual inverse of a matrix, and second---the ``Hadamard inverse'', that is, inverting each matrix entry separately.

\section{Absense of similar cocycles for pentagon, enneagon and hendecagon}\label{s:o}

\subsection[$(2n+1)$-gon relations for other $n$]{$\boldsymbol{(2n+1)}$-gon relations for other $\boldsymbol{n}$}\label{ss:n}

It was announced in~\cite[Subsection~6.3]{hepta_1} that our heptagon relation can be generalized to any $(2n+1)$-gon, $n=2,3,\ldots$. The \emph{form} of such relations has been explained in detail in~\cite[Section~II.B]{DK} (while \emph{different} $(2n+1)$-gon relations of the same form were presented in~\cite[Section~III]{DK}).

Here we want only to comment on the results of computer calculations for pentagon, enneagon (9-gon) and hendecagon (11-gon). This allows us just to say that all mentioned polygon relations do hold, and this can be checked directly using computer algebra, if we set the entries of the corresponding matrices~$A^{(p)}$ to have the following entries:
\begin{equation}\label{ansatz-g}
\left( A^{(p)} \right)_{ip}^{lp} = \prod_{\substack{j\ne i\\ jp\;\mathrm{inputs}}} \frac{d_{jlp}}{d_{ijp}}\,. 
\end{equation}
In greater detail: the product in~\eqref{ansatz-g} runs over all~$j$ such that the lower (input) legs of~$A^{(p)}$ are marked~$jp$, except $j=i$, see again Figure~\ref{fig:Ahepta} for the heptagon example. The determinants~$d_{ijk}$ are the same as in~\eqref{dijk}, except that matrix~$\mathcal M$~\eqref{Gelm} must have $2n+1$ columns (but still three rows).

Our definition~\eqref{ansatz-g} generalizes \eqref{ansatz-h}, of course. It works actually for all $(2n+1)$-gons, but the proof of this general fact (not relying on computer algebra) will appear elsewhere.

\subsection[$(2n-2)$-cocycles]{$\boldsymbol{(2n-2)}$-cocycles}\label{ss:2n-2}

Nontrivial $(2n-2)$-cocycles are given, for a general~$n$, by the very same formulas \eqref{4c} and~\eqref{cip}. The conceptual and general proof of this statement---a generalization of our Proposition~\ref{p:4c}---requires, however, some preparatory work, and we leave it for a future paper specifically addressing these issues. As far as only pentagon, enneagon and hendecagon relations are concerned, it can be checked quite easily using computer algebra.

Their uniqueness---the generalization of Proposition~\ref{p:4cnm}---also holds, and this time not much work is required. For clarity, we show how to do it on the example of enneagon. According to~\cite[Subsection~6.3]{hepta_1}, edge vectors are replaced, for enneagon, with \emph{triangle vectors}. Consider 7-simplex 12345678---also denoted simply~`9'---and triangle vectors for 123 and~456. The mentioned 7-simplex contains only two 6-faces, namely $1234567$ and $1234568$, also called $89$ and~$79$, containing both these triangles. Scalar product $\langle e_{123}, e_{456} \rangle_6^{(9)}$---the analogue of $\langle e_{12}, e_{34} \rangle_4^{(7)}$ from the proof of Proposition~\ref{p:4cnm}---must vanish, and this determines $c_{79} / c_{89}$ uniquely. Similarly, all ratios between coefficients $c_{ip}$ are also determined uniquely.

\subsection{Sequences of two coboundary operators for pentagon, enneagon and hendecagon}\label{ss:dd}

We now write out the analogues of sequence~\eqref{456} for pentagon, enneagon and hendecagon, assuming again the zero characteristic for field~$F$.

\paragraph{Pentagon}
\begin{equation}\label{234}
\begin{pmatrix} 10\\ \text{2-cochains} \end{pmatrix} \xrightarrow[\mathrm{rank} = 9]{\textstyle\delta} \begin{pmatrix} 15\\ \text{3-cochains} \end{pmatrix} \xrightarrow[\mathrm{rank} = 6]{\textstyle\delta} \begin{pmatrix} 6\\ \text{4-cochains} \end{pmatrix}
\end{equation}

\paragraph{Enneagon}
\begin{equation}\label{678}
\begin{pmatrix} 36\\ \text{6-cochains} \end{pmatrix} \xrightarrow[\mathrm{rank} = 35]{\textstyle\delta} \begin{pmatrix} 90\\ \text{7-cochains} \end{pmatrix} \xrightarrow[\mathrm{rank} = 55]{\textstyle\delta} \begin{pmatrix} 55\\ \text{8-cochains} \end{pmatrix}
\end{equation}

\paragraph{Hendecagon}
\begin{equation}\label{89A}
\begin{pmatrix} 55\\ \text{8-cochains} \end{pmatrix} \xrightarrow[\mathrm{rank} = 54]{\textstyle\delta} \begin{pmatrix} 165\\ \text{9-cochains} \end{pmatrix} \xrightarrow[\mathrm{rank} = 111]{\textstyle\delta} \begin{pmatrix} 120\\ \text{10-cochains} \end{pmatrix}
\end{equation}

The ranks of the \emph{left} operators~$\delta$ are always less than the number of $(2n-2)$-faces (and $(2n-2)$-cochains) by one, due to the reasons explained in the previous Subsection~\ref{ss:2n-2}.

The ranks of the \emph{right} operators~$\delta$ are again (like it was for sequence~\eqref{456}) more complicated, and were calculated using computer algebra and for generic parameters (entries of matrix~$\mathcal M$~\eqref{Gelm} analogues).

It follows from these ranks that there are no nontrivial cocycles in the middle terms, in a surprising contrast with the heptagon case!

\section{Discussion}\label{s:d}

Finally, some comments on possible directions of further research.

\paragraph{Other polygons and characteristics}
Rank calculations for the \emph{right} arrows in \eqref{234}, \eqref{678} and~\eqref{89A} were done only for the zero characteristic of field~$F$, and for a few lowest polygons. It is not known what awaits us outside these restrictions.

\paragraph{Homogeneous polynomial $\boldsymbol{(2n-1)}$-cocycles in finite characteristic from quadratic $\boldsymbol{(2n-2)}$-cocycles in characteristic~$\boldsymbol{0}$}
This can be done in the following four steps.
\begin{enumerate}
 \item Polarization: switch to the \emph{bilinear} form corresponding to a given quadratic cocycle. There appear thus \emph{two} permitted colorings, each enters linearly:
\[
c_{ij}x_{ij}^2 \mapsto c_{ij}x_{ij}y_{ij}.
\]
 \item Raise $x_{ij}$ into a degree~$p^k$, while $y_{ij}$ into a degree~$p^l$,  \ $k,l=1,2,\ldots$. It may call to mind Frobenius endomorphisms, but note that we are still in characteristic~$0$. We get polynomial \emph{cochain}
\[
c_{ij}x_{ij}^{p^k}y_{ij}^{p^l},
\]
homogeneous separately in $x$'s and~$y$'s.
 \item Take the coboundary. As we started with a \emph{cocycle}, the result is divisible by~$p$, other terms cancel out.
 \item Divide by~$p$ and reduce modulo~$p$, like it is done in a usual Bockstein homomorphism.
\end{enumerate}

\paragraph{From heptagon to hexagon with two-component colors}
Consider a 4-dimen\-sional Pachner move 3--3 and then the bicones over its lhs (initial configuration) and rhs (final configuration). Bicone means here the same as the join~\cite[Chapter~0]{Hatcher} with the boundary~$\partial I=\{0,1\}$ of the unit segment $I=[0,1]$. It is an easy exercise to see that the lhs bicone can be transformed into the rhs one by, first, a 5-dimen\-sional Pachner move 3--4 and, second, move 4--3. If there are now permitted colorings defined for the 4-faces of the 5-simplices involved, like in this paper, then we can attach two colors, or call it a \emph{two-component} color, to each 3-face of 4-simplices in the 3--3 move from which we started. It will be interesting to study connections of this construction with papers~\cite{cubic,nonconstant}.

\end{document}